\documentclass{amsart}

\usepackage{amsmath}
\usepackage{amsfonts}
\usepackage{amstext}
\usepackage{amsbsy}
\usepackage{amsopn}
\usepackage{amsxtra}
\usepackage{upref}
\usepackage{amsthm}
\usepackage{amsmath}
\usepackage{amssymb}
\usepackage{epsfig,verbatim,ifthen,graphicx}
\usepackage[all]{xy}
\newtheorem{prop}{Proposition}[section]
\newtheorem{lema}{Lemma}[section]
\newtheorem{teo}{Theorem}[section]

\theoremstyle{definition}

\def\R{{\mathbb R}}

\title{The Lyapunov spectrum as the Newton method.}
\date{\today}

\begin{thanks}
{This paper is dedicated to the memory of Sergio Plaza Salinas who worked in the interplay between dynamical systems and numerical analysis. His contribution to the development of dynamical systems in Chile is invaluable.  \\ 
The author was partially  supported by  Proyecto Fondecyt 1110040.}
\end{thanks}

\author{Godofredo Iommi} \address{Facultad de Matem\'aticas,
Pontificia Universidad Cat\'olica de Chile (PUC), Avenida Vicu\~na Mackenna 4860, Santiago, Chile}
\email{giommi@mat.puc.cl}
\urladdr{http://www.mat.puc.cl/\textasciitilde giommi/}

\begin{document}

\begin{abstract}
For  a class of dynamical systems, the  cookie-cutter maps, we prove that the Lyapunov spectrum coincides with the map given by the Newton-Raphson method applied to the derivative of the pressure function.
\end{abstract}

\maketitle

\section{Introdcution}
In this  note we establish a relation between the \emph{Newton-Raphson} method for finding roots of an equation and  an important function in the dimension theory of dynamical systems, namely the \emph{Lyapunov spectrum.}

The dimension theory of dynamical systems has received a great deal of attention over the last  years (see \cite{ba, pe}). Multifractal analysis is a sub-area of dimension theory devoted to study the complexity of level sets determined by  invariant local quantities. Usually, the geometry of the level sets is complicated and in order to quantify its size the Hausdorff dimension is used. In this note we will focus our attention in one of these local quantities, namely the Lyapunov exponent. Let $T:I \to I$ be a piecewise differentiable  map defined on the unit interval $I=[0,1]$. The  \emph{Lyapunov exponent} of  $T$  at the point $x \in I$, is defined by
\[ \lambda_T(x) = \lambda(x) = \lim_{n \to \infty} \frac{1}{n} \log |(T^n)'(x)|, \]
whenever the limit exists.  It is a dynamical quantity that measures the exponential rate of divergence of infinitesimally close orbits.   It is possible for the Lyapunov exponent to attain a wide range  of values. In order to understand the complexity of the induced level sets  we study the Lyapunov spectrum, which is the function defined by 
\begin{equation*}
L(\alpha):= \dim_H(\left\{ x \in I :  \lambda(x)= \alpha \right\}).
\end{equation*}
Here, $\dim_H$ denotes the Hausdorff dimension of a set (see \cite[Chapter 2]{fa} for a precise definition). This function can be described using thermodynamical formalism, in particular it can be written as a function of the topological pressure $P(-t \log |T'|)$, see Section \ref{ter} for  precise definitions and  statements.

We will relate the Lyapunov spectrum with the function used to apply the Newton-Raphson method. Given a real function $f:\R \to \R$, under certain assumptions,  the iterates of the Newton map $N_f(x)$ converge to a root of the equation $f(x)=0$ (for precise statements see section \ref{N}). The study of this function and its dynamic behaviour has attracted a great deal of attention (see \cite{mi}).

The main result in this note is that, for a certain class of dynamical systems, the Lyapunov spectrum equals the Newton map of the topological pressure $P(-t \log |T'|)$, that is
\begin{equation*}
L(  -P'(-t \log |T'|))= N_{P}(t).
\end{equation*}

\section{The Legendre transform and the Newton  method}

Let $f:\R \to \R$ be a strictly decreasing convex function of class $C^2$. In what follows we will define its corresponding Legendre transform and the associated Newton map. Moreover, we will show how are they related.

\subsection{Legendre transform}
Let $t ' \in \R$, the line $S_{\alpha}= \alpha(t-t') +f(t')$ is called a \emph{support line}
at $t'$ if it always lie below the graph of $f$.  There exists an interval of the form $[\alpha_{\min} , \alpha_{\max}]$ (this interval need not to be bounded) such that if $\alpha \in [\alpha_{\min} , \alpha_{\max}]$ the graph of $f$ has a support line $S_{\alpha}$ of slope equal to $-\alpha$ and for such $\alpha$ the support line is unique. The \emph{Legendre transform} of $f$ is the function $F:[\alpha_{\min} , \alpha_{\max}] \to \R$ defined as the value of the intercept of $S_{\alpha}$ with  the vertical axis. That is
\begin{equation} \label{leg}
F(\alpha)= \inf \{ f(t) + \alpha t : t \in \R \}.
\end{equation}
It turns out that if $t_{\alpha} \in \R$ is the unique point for which $f'(t_{\alpha})= -\alpha$ then
\begin{equation} \label{leg-1}
F(\alpha)= f(t_{\alpha}) + \alpha t_{\alpha}.
\end{equation}
Let us remark that in this simple setting the support line $S_{\alpha}$ is the tangent to the graph of $f$ at the point $(t_{\alpha}, f(t_{\alpha}))$ and that the interval $[\alpha_{\min} , \alpha_{\max}]$ is uniquely determined by the range of the derivative of $f$. 


Geometrically,  the Legendre transform is  the intercept of the tangent line of $f$ with  the vertical axis.

\subsection{The Newton-Raphson method} \label{N}

The Newton-Raphson method is an algorithmic method that provides a sequences of approximations to the root of the equation $f(t)=0$. The algorithm is based on the study of the dynamics of the so called \emph{Newton map}. This map is defined by \begin{equation} \label{new}
N_f(t)= t -\frac{f(t)}{f'(t)}.
\end{equation}
Note that since the function $f$ is strictly decreasing the Newton map is defined over the 
real numbers (the function $f$ has no critical values).  It turns out that a point $ d \in \R$ is a root of $f$ if and only if it is a fixed point for $N_f$. Moreover, since
\begin{equation} \label{der-n}
N'_f(t)= \frac{f(t) f''(t)}{(f'(t))^2},
\end{equation}
we have that $N'(d)=0$. In particular, the point $d$ is an attractor for  $N_f$ and iterations of the Newton map converge to the root $d$. That is, for any $x \in \R$ the sequence
$\{x, N_f(x), N^2_f(x), \dots , N_f^n(x), \dots\}$ converges to the root.

Geometrically, the procedure involved in this algorithm is to consider the tangent to the graph of $f$ at the point $(t, f(t))$ and define the Newton transformation as the intercept of the tangent with  the horizontal axis.

\subsection{A relation between the Legendre transform and the Newton method}
In virtue of the above we have that in order to compute the Legendre transform and the Newton map of $f$ we need to compute the tangent to $f$.
\begin{lema} \label{rel}
Let $f:\R \to \R$ be a strictly decreasing convex function of class $C^2$ and let  
$t_{\alpha} \in \R$ be a point such that $f'(t_{\alpha})= -\alpha$. Then
\begin{equation} \label{rel}
N_f(t_{\alpha}) =\frac{1}{\alpha}F(\alpha).
\end{equation}
\end{lema}
\begin{proof}
Note that since $t_{\alpha} \in \R$ is a point such that $f'(t_{\alpha})= -\alpha$ then in virtue of equation \eqref{new} the Newton map is given by
\[N_f(t_{\alpha})= t_{\alpha} +\frac{f(t_{\alpha})}{\alpha}\]
and the Legendre transform (see equation \eqref{leg-1}) is given by
\[F(\alpha)= f(t_{\alpha}) + \alpha t_{\alpha}.\]
The result now follows.
\end{proof}

\section{Thermodynamic formalism} \label{ter}
A large class of interesting dynamical systems have many invariant measures. It is, therefore, an important problem to find criteria to choose \emph{relevant} ones. \emph{Thermodynamic formalism} is a set of ideas and techniques which derive from statistical mechanics and that was brought into dynamics in the early seventies by Ruelle and Sinai among others (see the books \cite{Kellbook,Ru_book, Waltbook} for a detailed and nice exposition of the subject). This formalism  provides procedures for the choice  of interesting measures.  In this note we will consider the following class of dynamical systems. Given a pairwise disjoint finite family of closed intervals $I_1, \dots, I_n$
contained in $[0,1]$ we say that a map:
$$T: \bigcup_{i=1}^n I_i \to [0,1]$$
is a {\sf cookie-cutter map with $n$ branches} if the following holds:
\begin{enumerate}
\item $T(I_i)= [0,1]$ for every $i \in \{1 , \dots, n\}$,
\item The map $T$ is of class $C^{1 + \epsilon}$ for some $\epsilon >0$,
\item $|T' (x)| > 1$ for every $x \in I_1 \cup \cdots \cup I_n $.
\end{enumerate}
The dynamics is concentrated in the repeller $\Lambda$  of $T$, which defined by
\[ \Lambda:= \bigcap_{n=0}^{\infty} T^{-n} I.      \]
One of the most important objects in thermodynamic formalism is the so called  \emph{topological pressure} of $T$. It is defined by
\begin{equation} \label{pre}
P(t):=\sup \left\{ h(\mu) - t \int \log |T'| \ d\mu : \mu \in \mathcal{M}  \right\},
\end{equation}
where  $\mathcal{M}$ denotes the space of $T-$invariant probability measures and $h(\mu)$ the entropy of the measure $\mu$ (see \cite[Chapter 4]{Waltbook} for the definition and properties). For each $t \in \R$ there exists a unique measure $\mu_t$ attaining the supremum in equation \eqref{pre}  that we call \emph{equilibrium measure}
(see \cite[Chapter 5]{PU}).  It is well known that the pressure function $t \to P(t)$ is real analytic, strictly convex, strictly decreasing  and that $P'(t)= -\int \log |T'| \ d \mu_t$ (see \cite[Chapter 3,4]{pp} and \cite[Chapter 5]{PU}).

\subsection{Newton map}
A classical result by Bowen \cite{bo} later generalised by Ruelle \cite{ru} establishes a relation between the topological pressure and the Hausdorff dimension of the repeller.

\begin{prop}[Bowen's formula]
Let $T$ be a cookie cutter map and let $d=\dim_H \Lambda$ then $d$ is the unique root of the equation
\[P(-t \log |T'|)=0.\]
\end{prop}
Let us denote by $N_P(t)$ the Newton map of the pressure function $P(-t \log |T'|)$, that is
\[N_P(t)=  t - \frac{P(-t \log |T'|)}{P'(-t \log |T'|)}.\]
We have that  $N_P(d)=d$ and  the iterations of $N_P(t)$ converge to the unique root of the pressure, that is
\[\lim_{n \to \infty} N_P^n(t)=d.\]

\subsection{Lyapunov spectrum} \label{lya}

The Lyapunov exponent of any  given cookie cutter can attain an interval of values $[\alpha_{\min}, \alpha_{\max}]$.
For $\alpha \in [\alpha_{\min}, \alpha_{\max}]$ consider,
\begin{equation*}
J(\alpha)= \left\{ x \in I :  \lambda(x)= \alpha \right\} \text{ and } J'= \left\{  x \in I : \textrm{the limit } \lambda(x) \textrm{ does not exists} \right\}.
 \end{equation*}
These level sets induce a decomposition of the repeller $\Lambda$  of $T$. Indeed, $\Lambda = J' \cup \left( \cup_{\alpha} J(\alpha) \right).$ The function that encodes this decomposition is called \emph{Lyapunov spectrum} and it is defined by 
\begin{equation*}
L(\alpha):= \dim_H(J(\alpha)).
\end{equation*}
H. Weiss \cite{w1}, building up on previous joint work with Pesin \cite{pw1}, proved that the Lyapunov spectrum is real analytic. A rather surprising result in light of the fact that this decomposition is fairly complicated. For instance,  each level set turns out to be dense in $\Lambda$ (see \cite{w1}). By means of the thermodynamic formalism it is possible to obtain formula for the Lyapunov spectrum. Indeed,
\begin{equation}
\label{WeissFormula}
L(\alpha) = \frac{1}{\alpha} \inf_{t \in \mathbb{R}} (P(-t \log |T'|) +t \alpha).
\end{equation}
This formula follows form the work of Weiss \cite{w1} and can be found explicitly, for instance, in the work of Kesseb\"ohmer and Stratmann \cite{ks} (see also \cite{ik} for some explicit computations). Actually, in this setting, the Lyapunov spectrum can be written as
\begin{equation*} 
L(\alpha)=  \frac{1}{\alpha} (P(-t_{\alpha} \log |T'|) +t_{\alpha} \alpha),
\end{equation*}
where $t_{\alpha}$ is the unique real number such that 
\[ -\dfrac{d}{d t} P(-t \log |T'|) \Big|_{t=t_{\alpha}}  = \alpha. \]
If we denote by $\mu_{\alpha}$  the unique equilibrium measure corresponding to the potential  $-t_{\alpha} \log |T'|$ we have that
\[\int \log |T'| \ d\mu_{\alpha} = \alpha.\]
Therefore,
\begin{equation}
\label{lyapunov}
L(\alpha)=  \frac{1}{\alpha} (P(-t_{\alpha} \log |T'|) +t_{\alpha} \alpha),
\end{equation}
After the substitution $\alpha = \alpha(t)$, equation~\eqref{lyapunov} becomes:
\begin{equation} 
\label{lt}
L(\alpha(t))=  \frac{1}{\alpha(t)} \left( P(-t  \log |T'|)  +t \alpha(t) \right).  
\end{equation}

\section{Main result}
In virtue of the comments  we have made so far, we can now establish a relation between the  Lyapunov spectrum of a cookie cutter map and the Newton map corresponding to the pressure function. 
\begin{teo}
Let $T$ be a cookie cutter map, then for every $t \in \R$ we have 
\begin{equation}
 L(  -P'(-t \log |T'|))= N_{P}(t).
 \end{equation}
\end{teo}

\begin{proof}
The result follows noticing that the Lyapunov spectrum is given by (see equation \eqref{lt})
\begin{equation*} 
L(-P'(-t\log |T'|))=  \frac{1}{-P'(-t\log |T'|)} \left( P(-t  \log |T'|)  -t P'(-t\log |T'|) \right).  
\end{equation*}
This combined with Lemma \ref{rel} yields the desired result.
\end{proof}


%


%




%


%

\end{document}